\documentclass{article}
\usepackage{amsmath,amssymb,amsthm,amsfonts}
\usepackage{geometry}
\usepackage{braket}
\usepackage{algorithm}
\usepackage{algpseudocode}
\usepackage{xspace}
\usepackage{multirow}
\usepackage{graphicx}
\usepackage{makecell}
\usepackage{wrapfig}
\usepackage{array,tabularx}
\usepackage{diagbox} % For diagonal lines in table cells
\usepackage{lipsum}
\usepackage{placeins} % put this in your pre-amble
\usepackage{flafter}  % put this in your pre-amble
\usepackage[numbers,sort&compress]{natbib}
\usepackage[toc,page]{appendix}
\usepackage{comment}
\usepackage{mathtools}
\usepackage{booktabs}
\usepackage{authblk}
\usepackage{xcolor} 

\allowdisplaybreaks

\newcommand{\cA}{\mathcal{A}}

\newcommand{\cD}{\mathcal{D}}

\newcommand{\cH}{\mathcal{H}}

\renewcommand{\S}{\mathbb{S}}
\renewcommand{\H}{\mathbb{H}}
\newcommand{\R}{\mathbb{R}}
\newcommand{\C}{\mathbb{C}}

\renewcommand{\Re}{\operatorname{Re}}
\renewcommand{\Im}{\operatorname{Im}}

\newcommand{\argmin}{\mathrm{argmin}}

\newcommand{\tr}{\mathrm{Tr}}

\newcommand{\CVX}{\textsc{CVX}\xspace}
\newcommand{\Manopt}{\textsc{Manopt}\xspace}

\newtheorem{lemma}{Lemma}

\newtheorem{proposition}{Proposition}

\title{Solving cluster moment relaxation with hierarchical matrix}
\date{\today}

\author[1]{Yi Wang*}
\author[2]{Rizheng Huang\footnote{Equal contributions.}}

\author[3]{Yuehaw Khoo}
\affil[1]{Department of Statistics, University of Chicago}
\affil[2]{Committee on Computational and Applied Mathematics, University of Chicago}
\affil[3]{Department of Statistics and Committee on Computational and Applied Mathematics, University of Chicago}
\begin{document}
	\maketitle
	
	\begin{abstract}
		Convex relaxation methods are powerful tools for studying the lowest energy of many-body problems. By relaxing the representability conditions for marginals to a set of local constraints, along with a global semidefinite constraint, a polynomial-time solvable semidefinite program (SDP) that provides a lower bound for the energy can be derived. In this paper, we propose accelerating the solution of such an SDP relaxation by imposing a hierarchical structure on the positive semidefinite (PSD) primal and dual variables. Furthermore, these matrices can be updated efficiently using the algebra of the compressed representations within an augmented Lagrangian method. We achieve quadratic and even near-linear time per-iteration complexity. Through experimentation on the quantum transverse field Ising model, we showcase the capability of our approach to provide a sufficiently accurate lower bound for the exact ground-state energy.\end{abstract}
	
	\section{Introduction}
	
	Determining the lowest energy state of a many-body system is one of the most fundamental problems in science and engineering. This type of problem arises in the study of the Ising model \cite{cipra1987introduction}, graphical modeling \cite{wainwright2008graphical}, sensor network localization \cite{mao2007wireless}, and the structure from motion problem \cite{ozyesil2017survey}, to name a few. In these problems, one is usually concerned with minimizing an energy function  $\mathcal{E}$. With the exception of some simple cases, the energy landscape of $\mathcal{E}$ is plagued with spurious local minima. Without loss of generality, one can recast the problem of minimizing $\mathcal{E}$ as another equivalent minimization problem \cite{lasserre2001global}
	\begin{equation}\label{eq:measure optimization}
	E_0 = \inf_{\mu\in \mathcal{P}(\mathcal{S})} \int \mathcal{E}(x)\mu(dx)
	\end{equation}
	on the space of measure over a set $\mathcal{S}$, denoted as $\mathcal{P}(\mathcal{S})$. By moving to the space of measure, one effectively obtains a linear optimization problem that circumvents the non-convexity issue in minimizing $\mathcal{E}$, at the expense of dealing with the high-dimensional measure $\mu$. Generically, the solution of \eqref{eq:measure optimization} is an extreme point of $\mathcal{P}(\mathcal{S})$ which is a Dirac measure, and the support of such measure gives a minimizer of $\mathcal{E}$. Such a view is adopted in \cite{lasserre2001global,parrilo2003semidefinite} when devising a moment-based convex program for the case when $\mathcal{E}$ is a low-degree polynomial.

	There is an analogous problem in quantum many-body physics, where the ground-state energy minimization problem
	\begin{equation}\label{eq:measure quantum}
	E_0 = \min_{\rho} \; \tr(\hat{H}\rho), \quad \text{s.t. } \rho \ \text{is positive semidefinite}, \ \tr(\rho) = 1, 
	\end{equation}
	is commonly solved (for example the quantum Ising and Hubbard model \cite{altland2010condensed}). Here, $ \rho$ is a \emph{density operator} subject to certain constraints, and $\hat{H}$ is a high-dimensional Hamiltonian operator capturing the interactions between $d$-sites \cite{altland2010condensed}. The difficulty of solving \eqref{eq:measure quantum} is that the matrix $\rho$ scales exponentially as the number of bodies grows.
	
	\subsection{Prior works}
	The issue with measure or density operator minimization is that minimizing these high-dimensional objects is prohibitively expensive due to the curse of dimensionality. Therefore, instead of working with the high-dimensional measure or density operator, approaches based on moments have been proposed to solve \eqref{eq:measure optimization} and \eqref{eq:measure quantum} without the curse of dimensionality. \cite{lasserre2001global} proposes the use of moments to solve \eqref{eq:measure optimization}:
	\begin{equation}\label{eq:classic moment}
	m_{\alpha} := \mathbb{E}_\mu(x^\alpha),\quad x = (x_1,\ldots,x_d),\quad \alpha = (\alpha_1,\ldots,\alpha_d),
	\end{equation}
	where $\alpha\in \{0\cup \mathbb{N}\}^d$ is a multi-index, and $x^\alpha := x_1^{\alpha_1}\ldots x_d^{\alpha_d}$. A \emph{convex relaxation} is applied to  the space of $\{m_\alpha\}_\alpha$, where an outer approximation to the set of valid moments is given by a convex semidefinite program (SDP). Suppose we place the limit $\sum_i \alpha_i \leq q$ on the degree of the moments, we have a convex problem in terms of ${d +q \choose q}$ moments. Although by increasing $q$, the truncation threshold for the moments, one can improve the solution quality or even exactly recover the minimizer of the polynomial $p$ \cite{nie2014optimality}, for most practical situations, one can only use $q=2$ due to a large $d$.

	An analogous fermionic quantum mechanical version of the moment-based relaxation is detailed in \cite{2rdm_1,2rdm_3}. There, one deals with quantum moments of the form
	\begin{equation}
	m_\alpha := \text{Tr}(O^\alpha \rho),\quad O=(O_1,\ldots,O_d) 
	\end{equation}
	where $O^\alpha := O_1^{\alpha_1}\ldots O_d^{\alpha_d}$. Again, such a method can scale badly with $d$, constraining its application to small systems.

	To improve the scaling of these methods, recently, cluster moments/marginals semidefinite programming relaxations have been proposed, both for minimizing classical \cite{an1988note,pelizzola2005cluster,wainwright2008graphical} and quantum energies  \cite{variational_embedding_qmb, variational_embedding_pd_algorithm}. The general idea is that one first clusters the variables/operators, and only forms higher-order moments for intra-cluster variables/operators. This significantly lowers the number of decision variables involved, resulting in $O(d)\times O(d)$ type scaling. 
	
	In this paper, we adopt a strategy similar to the variational embedding method \cite{variational_embedding_qmb}, where we try to determine local cluster moments and combine them through a global PSD constraint. The difference is that the local cluster density matrices are represented through their moments. In this case, the decision variable is a PSD moment matrix. The main point of this paper is to propose a method to accelerate the PSD optimization problem. Typically, the most computationally expensive step in such an optimization problem is the projection onto the PSD cone, which scales cubically. In \cite{variational_embedding_pd_algorithm}, translation invariance of the Hamiltonian is exploited to diagonalize the PSD matrix in the Fourier basis with linear time complexity. However, it is unclear how such computational scaling can be achieved for general systems.

	\subsection{Our contributions}
	We propose a convex relaxation, the \emph{cluster moment relaxation}, to solve the energy minimization problem in both classical and quantum settings. Furthermore, we introduce a specific form of hierarchical matrices, which differs from the conventional definition, to represent the primal and dual variables of the proposed semidefinite relaxation. 
	
	The key point is that the constraints of the proposed relaxation can be enforced efficiently using the algebra of hierarchical matrices. Within an augmented Lagrangian method, with a hierarchical dual PSD variable, the optimization can be carried out with quadratic per-iteration complexity. Additionally, if one assumes the primal moment matrix also takes the form of a hierarchical PSD matrix, near-linear per-iteration complexity can be achieved.

	\subsection{Organizations}
	In Section~\ref{section:cluster moment relaxation}, we detail a  convex relaxation framework to solve energy minimization problems. In Section~\ref{section:ALM}, we review the augmented Lagrangian method (ALM) for solving the proposed convex program. In Section~\ref{section:hierarcical dual ALM} and \ref{section:hierarchical primal and dual ALM}, we propose the use of hierarchical matrices to accelerate the ALM. In Section~\ref{section:numerics}, we demonstrate the efficacy of the method for a quantum spin model.  
	
	\subsection{Notations}
	We use $I_n$ to denote the identity matrix of size $n \times n$. Additionally, we use $0_{m\times n}$ to denote a zero matrix of size $m \times n$, and when the context is clear, we will omit $m$ and $n$. Furthermore, let $\S^n$ be the space of real symmetric matrices of size $n \times n$, and let $\S^n_{+}$ be the positive semidefinite matrices in $\S^n$. Similarly, let $\H^n$ be the space of Hermitian matrices of size $n \times n$, and let $\H^n_{+}$ be the positive semidefinite matrices in $\H^n.$ For any matrix $X$ in $\S^n$ or $\H^n$, we may also use $X \succeq 0$ to denote that $X$ is positive semidefinite.
	
	When discussing a matrix $A$, the notation $A(p,q)$ refers to its $(p,q)$-th entry. In a block matrix $A$, $A_{ij}$ represents its $(i,j)$-th block. We may occasionally use $A_{ij}(p,q)$ to denote the $(p,q)$-th entry of the $(i,j)$-th block. For a complex-valued matrix $A$, $\Re(A)$ and $\Im(A)$ denote its real and imaginary parts, respectively. For a linear operator $\cA$ on matrices or vectors, its adjoint is denoted by $\cA^*$. Lastly, for any positive integer $N$, we use $[N]$ to denote the set $\{1,2,\cdots,N-1,N\}$.

	\section{Proposed convex relaxations}\label{section:cluster moment relaxation}
	
	While there are many versions of the cluster moments/marginals approach to obtain a convex relaxation of the energy minimization problems, in this paper, we examine a specific kind that only has equality constraints besides a global positive semidefinite constraint. As we shall see, this formulation can be optimized efficiently by our proposed method. The convex relaxation is constructed out of the following ingredients: 
	\begin{enumerate}
		\item \textbf{Cluster basis}: We first form monomials of variables/operators and cluster them into $K$ different groups $\mathcal{C}_i\subset \{0\cup \mathbb{N}\}^d, i\in [K]$. Each $\mathcal{C}_i$ is assumed to have $C$ elements, i.e. $\vert C_i \vert = C$. These clusters of monomials are called the \emph{cluster basis}:
		\begin{itemize}
			\item \textbf{Classical}: $\mathbf{v}= [\mathbf{v}_i]_{i\in[K]}$, $\mathbf{v}_i = [x^{\alpha}]_{\alpha\in \mathcal{C}_i}$.
			\item \textbf{Quantum}: $\mathbf{v} = [\mathbf{v}_i]_{i\in[K]}$, $\mathbf{v}_i = [O^{\alpha}]_{\alpha\in \mathcal{C}_i}$.
		\end{itemize}
		\item \textbf{Product cluster basis}: We then take the cluster basis and form their products as follows:
		\begin{itemize}
			\item \textbf{Classical}: $\mathbf{v}\mathbf{v}^*$, and $\mathbf{v}_i\mathbf{v}_j^* = [x^\alpha {x^\beta}^*]_{\alpha\in \mathcal{C}_i,\beta\in \mathcal{C}_j}$.
			\item \textbf{Quantum}: $\mathbf{v}\mathbf{v}^*$, and $\mathbf{v}_i\mathbf{v}_j^* = [O^\alpha {O^\beta}^*]_{\alpha\in \mathcal{C}_i,\beta\in \mathcal{C}_j}$.
		\end{itemize}
		\item \textbf{Intra-cluster relationship}: The cluster basis and the product of the basis elements within the same cluster satisfy the following linear constraint for $j \in [K]$:
		\begin{itemize}
			\item $\mathcal{D}_U(\mathbf{v}_j)+\mathcal{D}_L({\mathbf{v}_j}^*)+\mathcal{D}(\mathbf{v}_j\mathbf{v}_j^*) = \mathbf{z}$. $\mathbf{z}$ is a vector of scalars (operators) in the classical (quantum) case.
		\end{itemize}
		\item \textbf{Inter-cluster relationship}: The products of the basis elements, $\mathbf{v}_i \mathbf{v}_j^*$ and $\mathbf{v}_j \mathbf{v}_i^*$, for $i< j, i,j\in[K],$ satisfy the following relationship:
		\begin{itemize}
			\item $\mathcal{A}_U(\mathbf{v}_i\mathbf{v}^*_j)+\mathcal{A}_L(\mathbf{v}_j\mathbf{v}^*_i) = \mathbf{w}$. $\mathbf{w}$ is a vector of scalars (operators) in the classical (quantum) case.
		\end{itemize}
	\end{enumerate}

	The meaning of item 3 and 4 will be illustrated through an example in the next subsection. We now use these four ingredients to provide a convex relaxation for the         energy minimization problem in terms of the moment matrix $M \in 
        \mathbb{C}^{(CK+1) \times (CK +1)}$
	\begin{equation}\label{eq:moments definition}
	\text{Classical}: M = \mathbb{E}_\mu \left(\begin{bmatrix}\mathbf{v}\\ 1\end{bmatrix}\begin{bmatrix}\mathbf{v}^* & 1\end{bmatrix}\right),\quad \text{Quantum}: M = \mathbb{E}_\rho \left(\begin{bmatrix}\mathbf{v}\\ I\end{bmatrix}\begin{bmatrix}\mathbf{v}^* & I\end{bmatrix} \right).
	\end{equation}
	Here, the expectations are taken entry-wise and defined as  $\mathbb{E}_\mu(\cdot) := \int (\cdot) \mu(dx)$ and  $\mathbb{E}_\rho(\cdot) := \langle \cdot, \rho\rangle$. In what follows, for any matrix $A\in \mathbb{C}^{(CK+1) \times (CK +1)}$, we partition $A$ into
	\begin{equation}\label{A_matrix_split_moments}
	A := \begin{bmatrix} A^{(2)} & A^{(1)} \\ {A^{(1)}}^* & A^{(0)}\end{bmatrix},\quad  A^{(2)}\in \mathbb{C}^{CK \times CK},\ A^{(1)}\in \mathbb{C}^{CK \times 1},\ A^{(0)}\in\mathbb{C}.
	\end{equation}
	We often also write 
	\begin{equation}\label{eq:A block form}
	A^{(2)} = [A^{(2)}_{ij}]_{i=1,j=1}^K,\quad A^{(1)}=[A_i^{(1)}]_{i=1}^K
	\end{equation}
	where each $A^{(2)}_{ij}$ is a $C\times C$ block, and each $A_i^{(1)}$ is a $C\times 1$ vector. For example, $M^{(2) } = \mathbb{E}_\mu(\mathbf{v}\mathbf{v}^*)$ or $M^{(2)} = \mathbb{E}_\rho(\mathbf{v}\mathbf{v}^*)$.

	These ingredients give a set of necessary conditions on the moment matrix:
	\begin{enumerate}
		\item The inter-cluster relationship gives $\mathcal{A}_U(M^{(2)}_{ij})+\mathcal{A}_L(M^{(2)}_{ji}) = w$ for $i< j, i,j\in [K]$, where $w:=\mathbb{E}_{\mu,\rho}(\mathbf{w})$. 
		\item 
		The intra-cluster relationship gives $\mathcal{D}_U(M_j^{(1)})+\mathcal{D}_L({M_j^{(1)}}^*)+ \mathcal{D}(M_{jj}^{(2)}) = z$ for $ j\in[K]$, where $z:=\mathbb{E}_{\mu,\rho}(\mathbf{z})$. Furthermore, $M^{(0)} = 1$, since $\mathbb{E}_{\mu,\rho}(1) = 1$. 
		\item $M\succeq 0$.
	\end{enumerate}
	For convenience, we assume $w$ and $z$ are real-valued vectors of sizes $P_1$ and $P_2$, respectively. With these necessary conditions, an energy minimization problem can be relaxed into a convex problem as follows:
	\begin{eqnarray}
	(\textbf{P})&\ & \min_{M} \; \text{Tr}(JM)\label{eq:general moment relaxation}\cr
	\text{s.t.} &\ &\Lambda_{ij}\in \mathbb{R}^{P_1 \times 1} \hspace{0.5mm}:\mathcal{A}_U(M^{(2)}_{ij})+\mathcal{A}_L(M^{(2)}_{ji}) = w, \ i< j, \ i,j\in [K],\cr 
	&\ & \lambda_{j}\in \mathbb{R}^{P_2 \times 1} \hspace{1.8mm}: \mathcal{D}_U(M_j^{(1)})+\mathcal{D}_L({M_j^{(1)}}^*)+ \mathcal{D}(M_{jj}^{(2)}) = z, \ j\in[K],\cr
	&\ &\gamma\in \mathbb{R} \hspace{10.25mm} : M^{(0)} = 1,\cr
	&\ &S\in \mathbb{H}_+^{CK+1} \hspace{1.3mm}: M\in \mathbb{H}_+^{CK+1}.
	\end{eqnarray}
	$J$ is a cost matrix that encodes the information of $\mathcal{\epsilon}$ and $\hat H$ in \eqref{eq:measure optimization} and  \eqref{eq:measure quantum} respectively. The precise meaning of the terms in \eqref{eq:general moment relaxation} such as $J,w,z$ are illustrated in the next subsection through the example of a quantum spin model. 
	
	\subsection{Example of quantum energy minimization}\label{sec:quantum energy minimization}

	In physics, it is often the case that we have Hamiltonians with only pairwise interactions, i.e., the loss function in \eqref{eq:measure quantum} is equal to 
	\begin{equation} \label{gs_energy_density_formulation1}
	\sum_{1 \leq i \leq N} \tr(\hat{H}_i \rho) + \sum_{1 \leq i<j \leq N} \tr(\hat{H}_{ij} \rho).
	\end{equation}
	where $\hat H_i$'s and $\hat H_{ij}$'s are effectively some one-variable and two-variable operators.
	
	One example that we study in this paper is the quantum spin-$\frac{1}{2}$ system. The basic building blocks of these Hamiltonians are the Pauli matrices and the $2$-dimensional identity matrix:
	\begin{equation} \label{Pauli-matrices}
	\sigma^x = \begin{pmatrix}
	0 & 1 \\
	1 & 0
	\end{pmatrix}, \quad
	\sigma^y = \begin{pmatrix}
	0 & -i \\
	i & 0   
	\end{pmatrix}, \quad
	\sigma^z = \begin{pmatrix}
	1 & 0 \\
	0 & -1
	\end{pmatrix}, \quad
	I_2 = \begin{pmatrix}
	1 & 0 \\
	0 & 1
	\end{pmatrix}.
	\end{equation}
	These matrices form a basis for the real vector space of Hermitian operators on $\mathbb{C}^2$. Furthermore, let $\sigma_i^{x/y/z}$ denote the operator obtained by tensoring $\sigma^{x/y/z}$ on the $i$-th site with identities $I_2$ on all other sites, i.e.,
	\begin{equation} \label{Pauli_matrices_Nsites}
	\sigma_i^{\alpha} := I_2^{\otimes (i-1)} \otimes \sigma^{\alpha} \otimes I_2^{\otimes (N-i)}, \ \alpha \in \{x,y,z\}, 
	\end{equation}
	and let $I:= I_2^{\otimes N}$. In this case, all $\hat{H}_i$ and $\hat{H}_{ij}$ can be written as:
	\begin{align*}
	%\label{Decomposition_Hi_Hij_basis}
	& \hat{H}_i = \sum_{\alpha \in \{x,y,z\}} a^{\alpha}_i \sigma_i^{\alpha}, \quad  1\leq i \leq N, \\
	& \hat{H}_{ij} = \sum_{\alpha, \beta \in \{x,y,z\}} a^{\alpha \beta}_{ij} \sigma_i^{\alpha} \sigma_j^{\beta}, \quad 1 \leq i < j \leq N,
	\end{align*}
	for some real constants $\{a_i^{\alpha}\}_{i, \alpha}$ and $\{a_{ij}^{\alpha \beta}\}_{ij, \alpha \beta}.$

	For example, the celebrated $1$-D transverse field Ising (TFI) model \cite{tfi}, a spin-$\frac{1}{2}$ quantum model, is defined by its Hamiltonian:
	\begin{gather}\label{Hamiltonian_TFI}
	\hat{H}_{\text{TFI}} = -h \sum_{i=1}^N \sigma_i^x -\sum_{i=1}^N \sigma_i^z \sigma_{i+1}^z,
	\end{gather}
	where the TFI model is assumed to have periodic boundary conditions, i.e., $\sigma_{N+1}^{x/y/z}$ should be identified with $\sigma_1^{x/y/z}$, and $h \in \R$ is a scalar parameter controlling the strength of the external magnetic field along the $x$ axis.

	\subsubsection{Cluster moment relaxation for the TFI}\label{section:TFI constraints}
	
	We go through the four ingredients needed to construct a cluster moment relaxation. For an $N$-spin system, we define 
	\begin{equation} \label{def_vn}
	\textcolor{blue}{\mathbf{v} := [\mathbf{v}_i]_{i=1}^N, \quad \mathbf{v}_i = \begin{bmatrix} {\sigma_i^x}^* & {\sigma_i^y}^* & {\sigma_i^z}^* \end{bmatrix}^*,}
	\end{equation}
	  and $\mathbf{v}$ is a vector of operators of length $3N$. In this case, $K=N$ and $C=3$. Then we let $M$ of size $(3N+1) \times (3N+1)$ be defined by
	\begin{equation}
	M^{(2)}_{ij} =  \tr(\mathbf{v}_i {\mathbf{v}_j}^* \rho),\quad  M^{(1)}_i =  \tr(\mathbf{v}_i\rho).
	\end{equation}
	The properties of Pauli operators give rise to inter-cluster and intra-cluster constraints, which are summarized in the second column of Table~\ref{Table_linear_constraint_M}.
	\begin{table}[h]
		\centering
		\begin{tabular}{ |c|c| } 
			\hline
			& Relations between the operators  \\ 
			\hline
			\makecell{Inter-cluster \\ relationship} &\makecell{\rule{0pt}{3ex} \vspace{-0.1em} $\sigma_i^{\alpha} {\sigma_j^{\beta}}^*= {\sigma_j^{\beta}} {\sigma_i^{\alpha}}^*$, \\ $\alpha,\beta  \in \{x,y,z\},$ $i\neq j\in[N]$  \vspace{0.2em}}   \\   
			\hline
			\makecell{Intra-cluster \\ constraints 1} &\makecell{\rule{0pt}{3ex} \vspace{-0.1em} $\sigma_j^{\alpha} = {\sigma_j^{\alpha}}^*$, \\ $\alpha \in \{x,y,z\},$ $j\in[N]$ \vspace{0.2em}}   \\   
			\hline 
			\makecell{Intra-cluster \\ constraints 2}& \makecell{\rule{0pt}{3ex} \vspace{-0.2em} $\sigma_j^\alpha {\sigma_j^\beta}^* = i\sigma_j^\gamma $,  \\ $(\alpha,\beta,\gamma)=(x,y,z),(y,z,x),(z,x,y), j\in [N]$ \vspace{0.2em}}  \\
			\hline 
			\makecell{Intra-cluster \\ constraints 3} &\makecell{\rule{0pt}{3ex}$\sigma_j^{\alpha} {\sigma_j^{\alpha}}^*= I$, \\
				$ \alpha \in \{x,y,z\},  j\in [N]$}   \\
			\hline
		\end{tabular}
		\caption{Details of the basis property and inter/intra-cluster constraints}
		\label{Table_linear_constraint_M}
	\end{table}
	These relationships between operators then give constraints on the moment matrix
	\begin{enumerate}
		\item $i< j, \ \ i,j\in [K],$
            \begin{equation} \label{constraints:1}
                \mathcal{A}_U(M^{(2)}_{ij})+\mathcal{A}_L(M^{(2)}_{ji}) = w \Rightarrow \Im(M^{(2)}_{ij}) = 0.
            \end{equation} 
		\item $j\in[K], \ \  \mathcal{D}_U(M_j^{(1)})+\mathcal{D}_L({M_j^{(1)}}^*)+ \mathcal{D}(M_{jj}^{(2)}) = z  \Rightarrow$  
		\begin{eqnarray} \label{constraints:2}
		\Im(M_j^{(1)}) &=& 0 \cr 
		M^{(2)}_{jj} - \text{diag}(\{M^{(2)}_{jj}(k,k)\}^3_{k=1}) &=& \begin{bmatrix} 0 & iM^{(1)}_{j}(3)  & -iM^{(1)}_{j}(2) \\ -iM^{(1)}_{j}(3) & 0 & iM^{(1)}_{j}(1) \\ iM^{(1)}_{j}(2) & -iM^{(1)}_{j}(1) & 0\end{bmatrix}\cr 
		M^{(2)}_{jj}(k,k) &=& 1,\ k=1,2,3.\cr 
		\end{eqnarray}
	\end{enumerate}

	For the precise definitions of $\mathcal{A}_U$, $\mathcal{A}_L$, $\mathcal{D}_U$, $\mathcal{D}_L$, $\mathcal{D}$, and  $J,w,z$, we refer readers to the appendix.

	\section{Standard augmented Lagrangian method}\label{section:ALM}
	
	In this section, we first derive a standard ALM \cite{alm_intro, alm1, alm2} to solve \eqref{eq:general moment relaxation 2}, which serves as a foundation to discuss algorithms with reduced complexities in subsequent sections. The dual problem of (\textbf{P}) can be written as
	\begin{eqnarray}
	(\textbf{D})& & \underset{\substack{\{\Lambda_{ij}\}_{i<j},\{\lambda_j\}_j,\\ \gamma,S\in\mathbb{H}_+^{CK+1}}}{\max}\ \sum_{i<j\leq K} \langle \Lambda_{ij},  w \rangle+\sum_{j\in[K]} \langle \lambda_j, z\rangle + \gamma\label{eq:general moment relaxation 2}\cr
	\text{s.t.}&\ & J- \mathcal{F}^*(\{\Lambda_{ij}\}_{i<j},\{\lambda_j\}_j,\gamma)  = S
	%\begin{bmatrix} 0 & \Lambda_{12} & \cdots & \Lambda_{1K} \cr \Lambda_{12}^* & 0 & \ddots & \vdots \\ \vdots & \ddots & \ddots & \Lambda_{(K-1)K} \\ \Lambda^*_{1K} &\cdots & \Lambda_{(K-1)K}^*  & 0\end{bmatrix}
	\end{eqnarray}
	where 
	\begin{equation}\label{eq:simplified dual equality}
	\mathcal{F}^*(\{\Lambda_{ij}\}_{i<j},\{\lambda_j\}_j,\gamma)  := \begin{bmatrix}  \mathcal{D}^*(\lambda_1) & \mathcal{A}^*_U(\Lambda_{12}) & \cdots &\mathcal{A}^*_U(\Lambda_{1K}) &  & \mathcal{D}_U^*(\lambda_1) \\ \mathcal{A}^*_L(\Lambda_{12}) & \mathcal{D}^*(\lambda_2) & \ddots &  & & \vdots \\ \vdots & \ddots & \ddots & \mathcal{A}^*_U(\Lambda_{(K-1)K})& & \\ \mathcal{A}^*_L(\Lambda_{1K}) &\cdots & \mathcal{A}^*_L(\Lambda_{(K-1)K}) & \mathcal{D}^*(\lambda_K)  & & \mathcal{D}_U^*(\lambda_K) \\ \mathcal{D}_L^*(\lambda_1) & \cdots & & \mathcal{D}_L^*(\lambda_K) &  & \gamma \end{bmatrix}
	\end{equation} is hermitian so we have $\cA_U^*(\Lambda_{ij})=\left(\cA_L^*(\Lambda_{ij})\right)^*$ for $i<j$, and $\mathcal{D}^*(\lambda_j)=\left(\mathcal{D}^*(\lambda_j)\right)^*,\,\mathcal{D}_U^*(\lambda_j)=\left(\mathcal{D}_L^*(\lambda_j)\right)^*$ for $j\in[K]$.

	The dual problem (\textbf{D}) admits an augmented Lagrangian of the form
	\begin{multline} \label{augmented_lagrangian_dual_original_formulation}
	L_{\sigma} (S,\{\Lambda_{ij}\}_{i<j}, \{\lambda_j\}_j,\gamma; M) := -\sum_{i<j\leq K} \langle \Lambda_{ij}, w \rangle -\sum_{j\in[K]} \langle \lambda_j, z\rangle - \gamma \cr +  \frac{\sigma}{2} \left\|J - 
	\mathcal{F}^*(\{\Lambda_{ij}\}_{i<j},\{\lambda_j\}_j,\gamma)  - S-\sigma^{-1} M\right\|_F^2 - \frac{\|M\|_F^2}{2\sigma}
	\end{multline}
	with a penalty parameter $\sigma > 0$. Then, the ALM algorithm for solving (\textbf{D}) is summarized in Algorithm \ref{dual_problem_original_admm_algorithm}. 
	\begin{algorithm}
		\caption{ALM for the dual problem}\label{dual_problem_original_admm_algorithm}
		\begin{algorithmic}[1]
			\Require $S, \{\Lambda_{ij} \}_{i<j}, \{\lambda_j\}_j$, and $M$ satisfying linear constraints in (\textbf{P}), and penalty parameter $\sigma > 0$
			\While{\text{not converged}}
			\State $\{\Lambda_{ij}\}_{i<j}, \{\lambda_j\}_j,\gamma,S \hspace{1.5mm} \gets \underset{\substack{\{\Lambda_{ij}\}_{i<j},\{\lambda_j\}_j,\\ \gamma,S\in\mathbb{H}_+^{CK+1}}}{\arg\min}\ L_{\sigma} (S,\{\Lambda_{ij}\}_{i<j}, \{\lambda_j\}_j,\gamma; M) $ \label{dual_algorithm_joint_step}
			\State $M \hspace{3.8mm} \gets M + \sigma \left(S - J
			+\mathcal{F}^*(\{\Lambda_{ij}\}_{i<j},\{\lambda_j\}_j,\gamma)  \right)$ where $\mathcal{F}^*$ is defined in \eqref{eq:simplified dual equality} \label{dual_algorithm_M_step1}
			\EndWhile
		\end{algorithmic}
	\end{algorithm}
	In Algorithm \ref{dual_problem_original_admm_algorithm}, it is worth noting that the primal variable $M$ always satisfies the linear constraint in (\textbf{P}) during the updates. This arises from the property that ALM always provides dual-feasible variables \cite{alm_dual_feasible} (where (\textbf{P}) is the dual problem of (\textbf{D})). 
	
	In the next subsection, we detail how to perform the minimization sub-problem in Algorithm~\ref{dual_problem_original_admm_algorithm}. Indeed, since $\{\Lambda_{ij}\}_{i<j}, \{\lambda_{j}\}_{j}, \gamma$ are unconstrained, one can eliminate them from the minimization subproblem, leaving us with a minimization problem only in terms of $S$.

	\subsection{Minimization subproblem in Algorithm \ref{dual_problem_original_admm_algorithm}} \label{details_on_ALM_dual_problem}
	
	The subproblem in Step \ref{dual_algorithm_joint_step} of Algorithm \ref{dual_problem_original_admm_algorithm} is a joint minimization problem involving both $S$ and the rest of the dual variables. However, since the minimization subproblem is unconstrained with respect to $\{\Lambda_{ij}\}_{i<j}, \{\lambda_j\}_j, \gamma$, one can first eliminate them using the first-order optimality condition and express them in terms of $S$ (and $M$). We are then left with a minimization subproblem that involves only $S$. More precisely, 
	\begin{multline}\label{eq:Lambda update}
	\Lambda_{ij}(S^{(2)}_{ij},S^{(2)}_{ji}) = (\mathcal{A}_U \mathcal{A}_U^*+ \mathcal{A}_L \mathcal{A}_L^*)^{-1} \Big( \mathcal{A}_U (J^{(2)}_{ij} -S^{(2)}_{ij} - \sigma^{-1}M^{(2)}_{ij} )+ \cr \mathcal{A}_L(J^{(2)}_{ji} -S^{(2)}_{ji} - \sigma^{-1} M^{(2)}_{ji} ) + w/\sigma \Big), \  i<j, \ i,j\in [K],
	\end{multline}
	\begin{multline}\label{eq:lambda update}
	\lambda_{j}(S^{(2)}_{jj},S^{(1)}_{j},{S^{(1)}_{j}}^*) = ( \mathcal{D} \mathcal{D}^*+\mathcal{D}_U \mathcal{D}_U^* + \mathcal{D}_L \mathcal{D}_L^*)^{-1} \Big(\mathcal{D}(J^{(2)}_{jj}-S^{(2)}_{jj} -\sigma^{-1}M^{(2)}_{jj})+\\ \mathcal{D}_U(J^{(1)}_{j}-S^{(1)}_{j}-\sigma^{-1}M^{(1)}_{j})+\mathcal{D}_L({J^{(1)}_{j}}^*-{S^{(1)}_{j}}^*-\sigma^{-1}{M^{(1)}_{j}}^*)+ z/\sigma \Big), \ i\in [K],
	\end{multline}
	and
	\begin{equation}\label{eq:gamma update}
	\gamma(S^{(0)}) = J^{(0)}-S^{(0)} -\sigma^{-1}M^{(0)} +\sigma^{-1}.
	\end{equation}

	We now substitute the above dual variables into $L_\sigma$:
	\begin{equation}\label{eq:L with S}
	L^S_\sigma(S; M) := L_\sigma(S,\{\Lambda_{ij}(S^{(2)}_{ij},S^{(2)}_{ji})\}_{i<j}, \{\lambda_{j}(S^{(2)}_{jj},S^{(1)}_{j},{S^{(1)}_{j}}^*)\}_j,\gamma(S^{(0)}); M)
	\end{equation}
	and instead of minimizing $L_{\sigma} (S,\{\Lambda_{ij}\}_{i<j}, \{\lambda_j\}_j,\gamma; M)$, we can minimize $L^S_\sigma(S; M)$ as we have just eliminated all other dual variables from $L_{\sigma} (S,\{\Lambda_{ij}\}_{i<j}, \{\lambda_j\}_j,\gamma; M)$ except $S$ using first-order optimality conditions. The ALM solely in $S$ is summarized in Algorithm~\ref{dual_problem_original_admm_algorithm in S}.
	\begin{algorithm}
		\caption{ALM for the dual problem}\label{dual_problem_original_admm_algorithm in S}
		\begin{algorithmic}[1]
			\Require $S\in\mathbb{H}_+^{CK+1}$ and $M\in\mathbb{H}^{CK+1}$ satisfying linear constraints in (\textbf{P}), and penalty parameter $\sigma > 0$
			\While{\text{not converged}}
			\State $S \hspace{1.5mm} \gets \underset{S\in\mathbb{H}_+^{CK+1}}{\arg\min}\ L^S_{\sigma} (S; M) $ \label{dual_algorithm_onlyS_S_step}
			\State Update $\{\Lambda_{ij}\}_{i<j}$ using \eqref{eq:Lambda update}, $\{\lambda_{j}\}_{j}$ using \eqref{eq:lambda update}, and $\gamma$ using \eqref{eq:gamma update}
			\State $M \hspace{3.8mm} \gets M + \sigma \left(S - J
			+\mathcal{F}^*(\{\Lambda_{ij}\}_{i<j},\{\lambda_j\}_j,\gamma)  \right)$ where $\mathcal{F}^*$ is defined in \eqref{eq:simplified dual equality} \label{dual_algorithm_onlyS_M_step}
			\EndWhile
		\end{algorithmic}
	\end{algorithm}

	\section{ALM with hierarchical dual PSD variable}\label{section:hierarcical dual ALM}
	As shown previously, the ALM method (Algorithm~\ref{dual_problem_original_admm_algorithm in S}) requires minimizing \eqref{eq:L with S}, which is an optimization problem over the PSD cone. Typically, this requires  computing the projection onto the PSD cone, and the computational complexity of this projection is cubic, making it impractical for large-scale problems. 
	
	In \cite{BM-factorization}, the authors propose a solution by introducing a change of variables for the PSD variable in the form of $S = RR^*$. This strategy effectively circumvents the difficult PSD constraint in the optimization problem, converting it into an unconstrained optimization problem. Moreover, when low-rank solutions of the SDP problem exist, the number of columns of $R$ is chosen minimally, enabling the development of an efficient algorithm using the limited-memory BFGS algorithm. However, experiments conducted using the \CVX package \cite{CVX} to directly solve either the primal problem (\textbf{P}) or the dual problem (\textbf{D}) for the TFI model (Section~\ref{sec:quantum energy minimization}) indicate a linear increase in the rank in $N$ of both the primal PSD variable $M$ and the dual PSD variable $S$. Hence, employing a vanilla low-rank decomposition of $M$ or $S$ to solve either the primal or dual problem via the limited-memory BFGS algorithm is unlikely to yield substantial reductions in computation time. 
	
	In this section, we propose a structure for the dual variable $S$ in Algorithm~\ref{dual_problem_original_admm_algorithm} that allows us to perform the ALM updates with reduced time complexity. For a $1$-D TFI model (\ref{Hamiltonian_TFI}) with a small system size $N=128$ ($CK = 384$) and an external magnetic field strength parameter $h=1$, we solve (\textbf{D}) using an ADMM-type method with direct projection onto the PSD cone. We show a heatmap of the PSD variable $S$ in Figure \ref{heatmat_PSD_variable_128}. We can see from the plot that even though $S$ is not low-rank, it is nearly zero except on a few diagonals near the main diagonal. This observation inspired us to represent the dual PSD variable $S$ using a \emph{hierarchical low-rank matrix} \cite{hierarchical_matrix1}, resulting in an algorithm with quadratic scaling of the per-iteration time complexity. We emphasize that the matrix structure we propose is different from the typical hierarchical structure in the literature in order to encode the PSDness of the variable $S$. For simplicity, we assume that $K$ is a power of $2$ in this and the following sections.
	
	%\begin{figure}[t!] 
	\begin{figure}
		\centering
		\includegraphics[width=0.4\textwidth]{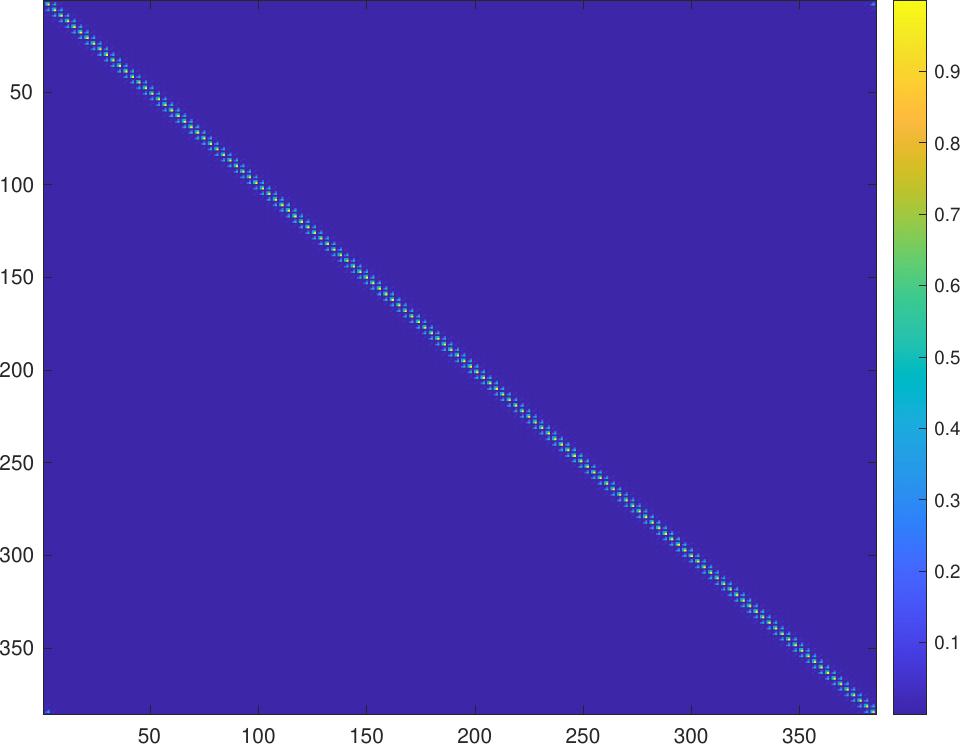}
		\caption{Heatmap of the dual PSD variable $S$ in (\textbf{D}) for $N=128$ TFI model.}
		\label{heatmat_PSD_variable_128}
	\end{figure}
	
	\subsection{Approximating $S$ with a hierarchical matrix} \label{subsection_definition_hierarchy}
	
	Let $S$ denote the solution of the dual problem (\textbf{D}) with size $(CK+1)\times (CK+1)$. We first examine the $S^{(2)}$ block of the matrix $S$ (as defined in \eqref{A_matrix_split_moments}). While our objective is to use a hierarchical matrix to represent $S^{(2)}$, we must also ensure that $S^{(2)}$ remains a positive semidefinite matrix. To this end, we use $m$ levels of hierarchy to characterize $S^{(2)}$. For the $l$-th level, we define a block diagonal matrix with $n_l$ diagonal blocks, where each block is of size $c_l\times c_l$. Furthermore, we want the block diagonal matrix to be positive semidefinite. Therefore, for the $l$-th level, we form a matrix:
	\begin{equation*} %\label{definition_of_Xi_hierarchy}
        \cH^{(2)}(y^{(l)}) :=
	\begin{pmatrix}
	y^{(l)}_1 (y^{(l)}_1)^* \\
	& y^{(l)}_2 (y^{(l)}_2)^* \\
	& & \ddots \\
	& & & y^{(l)}_{n_l} (y^{(l)}_{n_l})^* 
	\end{pmatrix} \in\H_+^{CK}, \quad y^{(l)}_j \in \C^{c_l \times r_l}, \quad 1 \leq j \leq n_l.
	\end{equation*}
	This naturally requires $n_l c_l = CK$, since the size of the matrix $S^{(2)}$ is $CK\times CK$. Then, as an approximation to $S^{(2)}$, we define
	\begin{multline} \label{approx_H2_hiearchy}
	\cH^{(2)}(y) := \cH^{(2)}(y^{(1)})+\cH^{(2)}(y^{(2)})+\cdots+\cH^{(2)}(y^{(m)})\\
	:=y^{(1)} (y^{(1)})^* + 
	\begin{pmatrix}
	y^{(2)}_{1} (y^{(2)}_{1})^* & 0 \\
	0 & y^{(2)}_{2} (y^{(2)}_{2})^* 
	\end{pmatrix} \\
	+\cdots+\begin{pmatrix}
	y^{(m)}_{1} (y^{(m)}_{1})^* & 0 & \cdots & 0 \\
	0 & y^{(m)}_{2} (y^{(m)}_{2})^* &  & 0 \\
	\vdots &  & \ddots & \vdots \\
	0 & \cdots & 0 & y^{(m)}_{n_m} (y^{(m)}_{n_m})^* \\
	\end{pmatrix},
	\end{multline}
	with 
	\begin{equation*}
	y := \big\{y^{(1)}, y^{(2)}, \cdots, y^{(m)} \big\}, \quad \text{and} \quad y^{(l)} := 
	\begin{pmatrix}
	y_1^{(l)} \\
	y_2^{(l)} \\
	\vdots \\
	y_{n_l}^{(l)}
	\end{pmatrix} \in \C^{CK \times r_l}.
	\end{equation*}
	Here, we assume $n_l = 2^{l-1}$ for $1 \leq l \leq m$ and $n_m = 2^{m-1} < K$. Our proposal involves representing $S^{(2)}$ with $\cH^{(2)}(y)$, where the number of levels $m$ and the number of columns for each level $r_l$ for $1 \leq l \leq m$ are determined based on the desired accuracy of the algorithm.

	To approximate the full matrix $S$, which has one extra row and one extra column compared to $S^{(2)}$, we just pad $\cH^{(2)}(y)$ with an extra row and column of zeros, plus a low-rank matrix:
	\begin{equation} \label{final_hierarchical_Hyz}
	\cH(y,t) := 
	\begin{pmatrix}
	\cH^{(2)}(y) & 0_{CK \times 1} \\
	0_{1 \times CK} & 0_{1 \times 1} 
	\end{pmatrix}+tt^* \approx S.
	\end{equation}
	Here, $y = \big\{y^{(1)}, \cdots, y^{(m)} \big\}$ is a set of matrices with $y^{(l)} \in \C^{CK \times r_l}$, and $t \in \C^{CK+1}$.
	
	%In the sections that follow, we require each $n_i$ to be a power of two, given by $n_i = 2^{i-1}$. In summary, the total number of variables needed for this hierarchical representation of $S$ is $\big(3N \sum_{i=1}^{m} r_i+(3N+1)r_z \big)$, and in situations where $\sum_{i=1}^{m} r_i + r_z \ll N$, the storage complexity required for the hierarchical representation of $S$ scales approximately linearly with the system size $N$. This new representation has a significant potential in reducing both computation time and required storage for the update of $S$. In the next section, we perform a numerical study on how well such a hierarchical structure can approximate $S$ for small systems.

	\subsubsection{Validity of the hierarchical matrix representation for $S$} \label{section:validity_S}
	We now investigate the validity of representing $S$ by a hierarchical matrix $\cH(y,t)$. For this purpose, we use the TFI model as a test problem and investigate the relationship between the system size $N$, the number of levels $m$ needed, and the number of columns needed for each $y^{(l)}$. We first solved the dual problem (\textbf{D}) with the Hamiltonian specified in (\ref{Hamiltonian_TFI}), with system sizes $N=64, 128, 256$, and an external magnetic field strength parameter $h=1$. For these problem sizes, we solved for $S$ in the full PSD cone rather easily with an accuracy of $10^{-4}$. Then, to see how well these PSD $S$ can be approximated by a hierarchical structure, we fitted the resulting PSD dual variable $S$ with the structure outlined in (\ref{final_hierarchical_Hyz}). Let $S^*$ denote the approximate solution to the dual problem (\textbf{D}). Additionally, let $y,t$ be complex-valued parameters that parameterize a hierarchical matrix of the corresponding size, with number of levels $m = 3, 4, 5$ respectively for $N= 64, 128, 256$. The number of columns for each level was fixed at $20$ for all system sizes. We solved the following optimization problem with system sizes $N=64,128,256$ for the variables $y$ and $t$:
	\begin{equation} \label{approx_error_hierarchy_S}
	\min_{y, t} \ \ \text{Err}_S^2 := \frac{\|\cH(y,t)-S^*\|_F^2}{\|S^*\|_F^2},
	\end{equation}
	by running the limited-memory BFGS algorithm provided in the \Manopt toolbox \cite{Manopt} for $100$ iterations. The approximation errors are presented in Table \ref{error_fitted_dual_S}. From the table, we can see that even with fixed $r_l = 20$ for all $1 \leq l \leq m$, we obtained similar accuracy for different system sizes. Therefore, we assume one can use a fixed rank approximation in the hierarchical matrix even for large system sizes.
	
	\begin{table}[h!]
		\centering
		\begin{tabular}{|c|c|c|c|} 
			\hline
			$N$ & 64 & 128 & 256 \\ [0.5ex] 
			\hline
			$\text{Err}_S$ & $2.1506e-06$ & $6.4785e-06$ & $4.5689e-05$ \\ [1ex] 
			\hline
		\end{tabular}
		\caption{Relative errors of the fitted dual PSD variables as defined in (\ref{approx_error_hierarchy_S})}
		\label{error_fitted_dual_S}
	\end{table}

	\subsection{Update rule with a hierarchically structured variable $S$} \label{computational_analysis_of_algo_one_hierarchy}

	By substituting the PSD variable $S$ in (\ref{augmented_lagrangian_dual_original_formulation}) with a data-sparse hierarchical PSD representation, we can eliminate the challenging PSD constraint on $S$, hence significantly reducing the per-iteration computational costs. With this hierarchical representation of $S$, when performing Algorithm~\ref{dual_problem_original_admm_algorithm} where one needs to minimize the variable-reduced augmented Lagrangian function $L^S_\sigma$, we replace $S$ with the hierarchical matrix $\cH(y,t)$ defined in (\ref{final_hierarchical_Hyz}). We remind the reader again $y$ is a collection of matrices $y := \big\{y^{(1)}, y^{(2)}, \cdots, y^{(m)} \big\}$ for $y^{(l)} \in \C^{CK \times r_l}$ and $t \in \C^{(CK+1) \times 1}$, with pre-specified number of levels $m$ and number of columns $r_1, \cdots, r_m$. The resulting algorithm is outlined in Algorithm~\ref{dual_problem_alm_hierarchies}.
	
	\begin{algorithm}
		\caption{Pseudo code for ALM for the dual problem with a hierarchical dual PSD variable}\label{dual_problem_alm_hierarchies}
		\begin{algorithmic}[1]
			\Require $y$, $t$,  $M \in \H^{CK+1}$ satisfying the linear constraint in (\textbf{P}), and penalty parameter $\sigma > 0$
			\While{\text{not converged}}
			\State $y,t \gets \argmin_{y,t} L^S_{\sigma}(\cH(y,t);M)$\ %$(\text{or}\ \argmin_{y,t} G_{\sigma}(\cH(y,t);M))$\label{algorithm_S_step1}
			\State $S \hspace{2.35mm} \gets \cH(y,t)$ \label{algorithm_S_step2}
			\State Update $\{\Lambda_{ij}\}_{i<j}$ using \eqref{eq:Lambda update}, $\{\lambda_{j}\}_{j}$ using \eqref{eq:lambda update}, and $\gamma$ using \eqref{eq:gamma update}
			\State  $M \hspace{1mm} \gets M + \sigma \left(S-J+ \mathcal{F}^*(\{\Lambda_{ij}\}_{i<j},\{\lambda_j\}_j,\gamma)\right)$ where $\mathcal{F}^*$ is defined in \eqref{eq:simplified dual equality} \label{algorithm_S_step4}
			\EndWhile
		\end{algorithmic}
	\end{algorithm}
	
	We now conduct a complexity analysis of Algorithm \ref{dual_problem_alm_hierarchies}, examining its computational scaling step by step.  To remind the readers, the Lagrangian  $L_\sigma$ can be split into three different terms:
	\begin{multline}\label{eq:second order loss}
	-\sum_{i<j} \langle \Lambda_{ij}, w \rangle + \sum_{i< j} \frac{\sigma}{2} \Big(\| J^{(2)}_{ij} - \mathcal{A}^*_U(\Lambda_{ij}) - S^{(2)}_{ij} - \sigma^{-1} M^{(2)}_{ij}\|_F^2 +\cr \| J^{(2)}_{ji} - \mathcal{A}^*_L (\Lambda_{ij}) - S^{(2)}_{ji} - \sigma^{-1} M^{(2)}_{ji}\|_F^2  \Big), 
	\end{multline}
	\begin{multline}\label{eq:first order loss}
	-\sum_{j} \langle\lambda_j,z\rangle  + \frac{\sigma}{2}\sum_{j}\| J^{(2)}_{jj} - \mathcal{D}^*(\lambda_j)- S^{(2)}_{jj} - \sigma^{-1} M^{(2)}_{jj}\|_F^2  \\
	+ \frac{\sigma}{2}\sum_j \| J^{(1)}_{j} - \mathcal{D}^*_U(\lambda_j)- S^{(1)}_{j} - \sigma^{-1} M^{(1)}_{j} \|_F^2 +\frac{\sigma}{2}\sum_j  \| {J^{(1)}_{j}}^* - \mathcal{D}^*_L(\lambda_j)- {S^{(1)}_{j}}^* - \sigma^{-1} {M^{(1)}_{j}}^* \|_F^2,
	\end{multline}
	and
	\begin{equation}\label{eq:zeroth order loss}
	-\gamma  + \frac{\sigma}{2}\vert J^{(0)} -\gamma- S^{(0)} - \sigma^{-1} M^{(0)}\vert^2, 
	\end{equation}
	and $L^S_\sigma$ is obtained by substituting \eqref{eq:Lambda update}, \eqref{eq:lambda update} and \eqref{eq:gamma update} into \eqref{eq:second order loss}, \eqref{eq:first order loss} and \eqref{eq:zeroth order loss} respectively.

	Suppose we use gradient-based methods such as the limited-memory BFGS algorithm for Step 2 in Algorithm~\ref{dual_problem_alm_hierarchies}. Since the complexity of computing the gradient is asymptotically the same as the complexity of evaluating the loss function \cite{gradeqloss}, we simply analyze the computational cost of evaluating $L^S_\sigma$. The key operations in evaluating the loss $L^S_\sigma$ consist of evaluating the terms \eqref{eq:second order loss}, \eqref{eq:first order loss} and \eqref{eq:zeroth order loss}. Since \eqref{eq:first order loss} and \eqref{eq:zeroth order loss} have $O(CK)$ terms and $O(1)$ terms respectively, the computational complexity for these terms is negligible compared to $\eqref{eq:second order loss}$, which has $O(C^2 K^2)$ terms. When substituting \eqref{eq:Lambda update} into \eqref{eq:second order loss}, we have 
	\begin{multline}\label{eq:second order loss sub}
	-\sum_{i<j} \langle \Lambda_{ij}(S^{(2)}_{ij},S^{(2)}_{ji}), w \rangle + \sum_{i< j} \frac{\sigma}{2} \Big(\| J^{(2)}_{ij} - \mathcal{A}^*_U(\Lambda_{ij}(S^{(2)}_{ij},S^{(2)}_{ji})) - S^{(2)}_{ij} - \sigma^{-1} M^{(2)}_{ij}\|_F^2 \\ +\| J^{(2)}_{ji} - \mathcal{A}^*_L(\Lambda_{ij}(S^{(2)}_{ij},S^{(2)}_{ji}) ) - S^{(2)}_{ji} - \sigma^{-1} M^{(2)}_{ji}\|_F^2  \Big). 
	\end{multline}
	The question now is, with $S = H(y,t)$, meaning $S^{(2)}$ is the sum of $H^{(2)}(y)$ plus a rank-one correction term, what is the complexity of evaluating \eqref{eq:second order loss sub}. As the rank-one correction term can also be assimilated into the hierarchical structure, by adding one more column to the first hierarchy $y^{(1)}$, we can, without loss of generality, assume $S^{(2)}$ is equal to a hierarchical matrix. From the construction of $w$ in the linear constraints in \eqref{constraints:1}, we know $w$ is a matrix of all zeros, so the term $\sum_{i<j} \langle \Lambda_{ij}(S^{(2)}_{ij},S^{(2)}_{ji}), w \rangle$ disappears. Additionally, the remaining two terms $\| J^{(2)}_{ij} - \mathcal{A}^*_U(\Lambda_{ij}(S^{(2)}_{ij},S^{(2)}_{ji})) - S^{(2)}_{ij} - \sigma^{-1} M^{(2)}_{ij}\|_F^2$ and $\| J^{(2)}_{ji} - \mathcal{A}^*_L (\Lambda_{ji}(S^{(2)}_{ij},S^{(2)}_{ji}) ) - S^{(2)}_{ji} -\sigma^{-1} M^{(2)}_{ji}\|_F^2$ take similar forms, so it suffices to just analyze the complexity of the first term.
 
    As $w$ is the zero matrix, from the definitions of the operator $\mathcal{A}^*_U$ and the optimal $\Lambda_{ij}(S^{(2)}_{ij},S^{(2)}_{ji})$ in \eqref{constraints:1} and \eqref{eq:Lambda update}, it is apparent that $\left( J^{(2)}_{ij} - \mathcal{A}^*_U(\Lambda_{ij}(S^{(2)}_{ij},S^{(2)}_{ji})) - S^{(2)}_{ij} - \sigma^{-1} M^{(2)}_{ij} \right)$ is a linear function of $J_{ij}^{(2)}, J_{ji}^{(2)}, S_{ij}^{(2)}, S_{ji}^{(2)}, M_{ij}^{(2)}$ and $M_{ji}^{(2)}$. Therefore, we need to evaluate a term of the form $$\sum_{i < j} \|\cA^<_S(S_{ij}^{(2)})+\cA^>_S(S_{ji}^{(2)})+\cA^<_J(J_{ij}^{(2)})+\cA^>_J(J_{ji}^{(2)})+\cA^<_M(M_{ij}^{(2)})+\cA^>_M(M_{ji}^{(2)})\|_F^2,$$ 
    where $A^<_S, A^<_J, A^<_M, A^>_S, A^>_J, A^>_M: \C^{C \times C} \rightarrow \C^{C \times C}$ are all linear operators. The complexity of such an evaluation, up to a constant independent of $S^{(2)}$, is dominated by the inner products between these terms, and Lemma~\ref{lemma:summary of computations} gives the complexity of these operations.
    
    It is important to note that although $w$ is assumed to be the zero matrix in our proposed relaxation, the complexity analysis can be conducted in a similar way as long as $w$ is a constant matrix. To illustrate this, let $W \in \C^{CK \times CK}$ be a matrix in which every $C \times C$ block matrix is equal to $w$. Consequently, $W$ is a matrix with rank less than or equal to $C$, and thus is also a hierarchical matrix. Therefore, we can apply Lemma~\ref{lemma:summary of computations} to analyze the complexity of the inner product between $S^{(2)}$ and $W$.
	
	\begin{lemma}\label{lemma:summary of computations}
		Assuming $C$ is a constant, the complexity of
		\begin{equation}\label{eq:inner product between two matrices}
		\sum_{i<j\leq K} \tr(B_{ij} \mathcal{A}(B'_{ij}))
		\end{equation}
		for some matrices $B,B'\in \mathbb{C}^{CK\times CK}$ and some linear operator $\mathcal{A}:\mathbb{C}^{C\times C}\rightarrow \mathbb{C}^{C\times C}$ is:
		\begin{enumerate}
			\item $O(K m^2 r^2)$ if both $B$ and $B'$ are hierarchical matrices in the form of \eqref{approx_H2_hiearchy} with $m$ levels and each level having rank $r$.
			\item  $O(K^2 r)$ if $B$ is an arbitrary matrix and $B'$ is a hierarchical matrix.  
			\item $O(Kmr)$ if $B$ is sparse with $O(K)$ non-zero entries and $B'$ is a hierarchical matrix.
		\end{enumerate}
	\end{lemma}
	\begin{proof}
		For any linear operator $\mathcal{A}:\mathbb{C}^{C\times C}\rightarrow \mathbb{C}^{C\times C}$, $\sum_{i<j\leq K} \tr(B_{ij} \cA(B'_{ij}))$ can be written as 
        \begin{equation} \label{lemma1_equa1}
            \sum_{\substack{i<j\\i,j\in [K]}} \sum_{k,\kappa,k',\kappa'\in [C]} D_{\cA}(k,k',\kappa,\kappa')B_{ij}(k,\kappa)B'_{ij}(k',\kappa')
        \end{equation} 
        for a $4$-tensor $D_{\cA}$ whose values depend on $\cA$. If both $B$ and $B'$ take the form in \eqref{approx_H2_hiearchy}, the sum \\ 
        $\sum_{\substack{i<j\\i,j\in [K]}}B_{ij}(k,\kappa)B'_{ij}(k',\kappa')$ can be computed with $O(Km^2 r^2)$ complexity (see Proposition~\ref{computational_complexity_hierarchical_inner_product}). Then the sum $\sum_{k,\kappa,k',\kappa'\in [C]}$ contributes a factor of $C^4$, giving a total complexity of  $O(C^4 K m^2 r^2)$. We ignore the factor $C^4$ as $C$ is assumed to be a constant. 
		
		The complexity of the second statement can be shown in a similar way. If $B'$ is a hierarchical matrix, using Proposition~\ref{prop:computational_complexity_hierarchical_inner_product_full}, one can show that the sum $\sum_{\substack{i<j\\i,j\in [K]}}B_{ij}(k,\kappa)B'_{ij}(k',\kappa')$ can be computed with $O(K^2 r)$ complexity. The summation $\sum_{k,\kappa,k',\kappa'\in [C]}$ further contributes a factor of $C^4$, which is again ignored as $C$ is assumed to be a constant.
		
		The last statement is a direct consequence of $B$ being a sparse matrix. 
	\end{proof}
	
	Based on this lemma, the complexity of $\sum_{i < j} \langle \cA^<_S(S_{ij}^{(2)}), \cA^<_J(J_{ij}^{(2)})+\cA^>_J(J_{ji}^{(2)}) \rangle$ and \\ $\sum_{i < j} \langle \cA^>_S(S_{ji}^{(2)}), \cA^<_J(J_{ij}^{(2)})+\cA^>_J(J_{ji}^{(2)}) \rangle$ is $O(Kmr)$, due to the fact that $J$ is sparse with $O(K)$ non-zero entries. The complexity of $\sum_{i<j}\|\cA^<_S(S_{ij}^{(2)})+\cA^>_S(S_{ji}^{(2)})\|_F^2$ is $O(Km^2r^2)$, and the complexity of $\sum_{i < j} \langle \cA^<_S(S_{ij}^{(2)}), \cA^<_M(M_{ij}^{(2)})+\cA^>_M(M_{ji}^{(2)}) \rangle$ and $\sum_{i < j} \langle \cA^>_S(S_{ji}^{(2)}), \cA^<_M(M_{ij}^{(2)})+\cA^>_M(M_{ji}^{(2)}) \rangle$ is $O(K^2 r)$. Assuming $r$ is a constant and $m = \log_2(CK)$, the computational cost is dominated by the inner products between blocks of $S$ and $M$ as it has a quadratic growth with respect to $K$. This stems from the fact that $M^{(2)}$ is an unstructured matrix. At this point, we have successfully reduced the per-iteration cost of vanilla ALM (Algorithm~\ref{dual_problem_original_admm_algorithm in S}) from $O(K^3)$ to $O(K^2)$ by assuming $S^{(2)}$ takes the form of a hierarchical positive semidefinite matrix in Algorithm~\ref{dual_problem_alm_hierarchies}.

	%We first examine the computational complexity when assuming $S^{(2)}$ being a rank-1 matrix, i.e. $S^{(2)} = xx^*, x\in \mathbb{C}^{CK}$. 

	\section{ALM with hierarchical primal and dual PSD variables}\label{section:hierarchical primal and dual ALM}
	
	As analyzed in Section \ref{computational_analysis_of_algo_one_hierarchy}, Algorithm~\ref{dual_problem_alm_hierarchies} has an $O(K^2)$ per-iteration complexity due to the lack of structure in the primal variable $M$. While this is already a speed-up compared to a vanilla ALM with cubic complexity, we propose replacing the direct update rule in  Step \ref{algorithm_S_step4} of Algorithm \ref{dual_problem_alm_hierarchies} with a projection step that compresses $M$ in order to obtain a nearly linear per-iteration cost. 
	
	Before discussing how to form a compressed representation for the primal variable, we rewrite the primal variable update in Algorithm \ref{dual_problem_alm_hierarchies} as the solution to the following problem:
	\begin{align} \label{Projection_problem_step1}
	\begin{split}
	&\argmin_{\tilde M} \quad \|\tilde M - \big(M + \sigma \left(S-J+ \mathcal{F}^*(\{\Lambda_{ij}\}_{i<j},\{\lambda_j\}_j,\gamma)\right))\big)\|_F^2, \\
	\end{split}
	\end{align}
	Here, $S$ is represented hierarchically as $S = \cH(y,t)$, with  $\{\Lambda_{ij}\}_{i<j}, \{\lambda_{j}\}_j, \gamma$ being defined in \eqref{eq:Lambda update}, \eqref{eq:lambda update} and \eqref{eq:gamma update}. For simplicity, we still use (25) to represent the primal variable $M$ as a hierarchical matrix, i.e. $M = \cH(x,v)$, where $x = \big\{x^{(1)}, \cdots, x^{(m)} \big\}$ is a set of matrices with $x^{(l)} \in \C^{CK \times r_l}$, and $v \in \C^{CK+1}$. Then \eqref{Projection_problem_step1} becomes:
	\begin{align} \label{Projection_problem_step2}
	\begin{split}
	&\argmin_{x,v} \quad \|\cH(x,v) - \big(M + \sigma \left(S-J+ \mathcal{F}^*(\{\Lambda_{ij}\}_{i<j},\{\lambda_j\}_j,\gamma)\right))\big)\|_F^2, \\
	\end{split}
	\end{align}

	We now introduce Algorithm~\ref{dual_problem_alm_two_hierarchies}, which utilizes a hierarchical representation for both the primal and the dual PSD variables.
	\begin{algorithm}[H]
		\caption{Pseudo code for ALM for the dual problem with two hierarchical PSD variables}
		\label{dual_problem_alm_two_hierarchies}
		\begin{algorithmic}[1]
			\Require $y,t$ for the dual variable $S=\cH(y,t)$, $x,v$ for the primal variable $M=\cH(x,v)$, and penalty parameter $\sigma > 0$
			\While{not converged}
			\State $M \hspace{1mm} \leftarrow \cH(x,v)$
			\State $y,t \gets \argmin_{y,t} L^S_{\sigma}(\cH(y,t);M)$
			\State $S \hspace{2.3mm} \gets \cH(y,t)$ \label{algorithm_final_step2}
			\State Update $\{\Lambda_{ij}\}_{i<j}$ using \eqref{eq:Lambda update}, $\{\lambda_{j}\}_{j}$ using \eqref{eq:lambda update}, and $\gamma$ using \eqref{eq:gamma update}
			\State Update $x,v$ by solving \eqref{Projection_problem_step2} \label{algorithm_final_step4}
			\EndWhile
		\end{algorithmic}
	\end{algorithm}
	We highlight that $M$ and $S$ are never explicitly formed as $(CK+1)\times (CK+1)$ matrices to ensure efficient computations. The per-iteration computational complexity of Algorithm  \ref{dual_problem_alm_two_hierarchies} can be analyzed following the approach in Section  \ref{computational_analysis_of_algo_one_hierarchy} for Algorithm \ref{dual_problem_alm_hierarchies}. We assume the hierarchical representation of the primal and dual PSD variables has the same number of levels $m$, with a constant number of columns $r$ for each level. As before, in Step~\ref{algorithm_S_step2} and Step~\ref{algorithm_S_step4} of the algorithm, the evaluation of the loss consists of computations in Lemma~\ref{lemma:summary of computations}. When we replace both primal and dual variables with hierarchical matrices, the second scenario in Lemma~\ref{lemma:summary of computations} is eliminated, and we achieve a near-linear complexity of $O(Km^2r^2)$ with an $r$ that does not grow with $K$ and $m = O(\log_2(CK))$.
	
	\subsection{Validity of the hierarchical matrix representation for M}
        In this section, we investigate the validity of representing $M$ by a hierarchical matrix $\cH(x,v)$, using the same method as in section \ref{section:validity_S} for the dual PSD variable $S$. Instead of fitting the dual PSD variable $S$, we fitted the primal PSD variable $M$ resulting from solving (\textbf{D}) with the Hamiltonian specified in (\ref{Hamiltonian_TFI}), for system sizes $N=64, 128, 256$, an external magnetic field strength parameter $h=1$ and an accuracy of $10^{-4}$. Let $M^*$ be the approximate solution. Let $x,v$ be complex-valued parameters that parameterize a hierarchical matrix of the corresponding
        size, with the number of levels $m = 3, 4, 5$ respectively for $N = 64, 128, 256$. The number of columns for each level was fixed at $20$ for all system sizes. We solved the following optimization problem
        with system sizes $N = 64, 128, 256$ for the variables $x$ and $v$:

        \begin{equation} \label{approx_error_hierarchy_M}
	\min_{x,v} \ \ \text{Err}_M^2 := \frac{\|\cH(x,v)-M^*\|_F^2}{\|M^*\|_F^2},
	\end{equation}
	by running the limited-memory BFGS algorithm provided in the \Manopt toolbox \cite{Manopt} for $100$ iterations. The approximation errors are presented in Table \ref{error_fitted_dual_M}. The experiment results indicate that as the system size increases, the relative error of the fitted primal PSD variable slowly increase, showing that the hierarchical matrix representation for $M$ can be valid for larger system sizes. 
	
	\begin{table}[h!]
		\centering
		\begin{tabular}{|c|c|c|c|} 
			\hline
			$N$ & 64 & 128 & 256 \\ [0.5ex] 
			\hline
			$\text{Err}_M$ & $3.2466e-04$ & $4.7392e-04$ & $6.1500e-04$ \\ [1ex] 
			\hline
		\end{tabular}
		\caption{Relative errors of the fitted primal PSD variables as defined in (\ref{approx_error_hierarchy_M})}
		\label{error_fitted_dual_M}
	\end{table}

	\section{Numerical experiments}\label{section:numerics}
	In this section, we present numerical experiments for the $1$-D TFI model using Algorithms \ref{dual_problem_alm_hierarchies} and \ref{dual_problem_alm_two_hierarchies}, with system sizes $N \in \{64,128,256,512,1024,2048,4096\}$. The penalty parameter is initialized at $\sigma = 0.1$ for $N\geq 1024$, $\sigma = 1$ for $N\leq 512$ and is adjusted dynamically \cite{admm_penalty_adjust} based on primal and dual feasibility to speed up the convergence of the ALM algorithm. For both algorithms, the number of levels $m$ in the hierarchy is set to be $m = 3,4,\cdots,9$ for $N=64,128,\cdots,4096$, and the number of columns for all levels is set to be $r=20$. In Algorithm~\ref{dual_problem_alm_hierarchies}, $y$, $t$ are initialized from the standard normal distribution, and $M_0=I_{3N+1}$. In Algorithm~\ref{dual_problem_alm_two_hierarchies}, $x$, $v$, $y$ and $t$ are randomly initialized from the standard normal distribution. Throughout the updates of Algorithms \ref{dual_problem_alm_hierarchies} and \ref{dual_problem_alm_two_hierarchies}, we evaluate the accuracy of approximate solutions by monitoring the relative primal feasibility, the relative dual feasibility, and the relative duality gap, as detailed in the following.
	
	The primal feasibility of the variable $M$ is governed by how well it satisfies the PSD and linear constraints in (\textbf{P}). In Algorithm~\ref{dual_problem_alm_hierarchies}, $M$ is directly updated as a dense matrix, while in Algorithm~\ref{dual_problem_alm_two_hierarchies}, $M$ is maintained using a compressed representation that satisfies the linear constraints. Since in Algorithm~\ref{dual_problem_alm_hierarchies}, $M$ is guaranteed to satisfy the linear constraints in (\textbf{P}), we only monitor its PSDness using the following measure:
	$$\eta_{P_3} := \frac{\max(0, -\lambda_{\text{min}})}{1+\max(0, \lambda_{\text{max}})},$$ where $\lambda_{\text{min}}$ and $\lambda_{\text{max}}$ are the smallest and largest eigenvalues of $M$. In Algorithm~\ref{dual_problem_alm_two_hierarchies}, we update $M$ by solving (34), which guaranteeds that $M$ is PSD. Therefore, we only monitor the linear constraint satisfaction by:
	$$\eta_{P_4} := \frac{\|\mathcal{A}(M)-b\|_2}{1+\|b\|_2},$$ 
	where $\mathcal{A}(M)=b$ encodes all linear constraints in $(\textbf{P})$.
	
	For the dual problem (\textbf{D}) and a candidate dual variable $S$, the PSD constraint for $S$ is automatically satisfied because $S$ is maintained as a positive semidefinite hierarchical matrix in both algorithms. We thus only monitor the dual feasibility by how well the dual equality constraint is satisfied using the following measure:
	\begin{equation*}
	\eta_{D}:= \frac{\|S - J
		+\mathcal{F}^*(\{\Lambda_{ij}\}_{i<j},\{\lambda_j\}_j,\gamma) \|_F}{1+\|J\|_F}.
	\end{equation*}
	
	Finally, we monitor the relative duality gap by:
	$$\eta_g := \frac{|\text{primal objective}-\text{dual objective}|}{1+|\text{primal objective}|+|\text{dual objective}|}.$$ 
	
	We terminate the algorithm when $\eta := \max(\eta_P, \eta_D, \eta_g) \leq 10^{-3}$, or when the ALM algorithm has run for $150$ iterations. It is important to highlight that we exploit the hierarchical structure present in the PSD primal and dual variables to efficiently evaluate these convergence metrics.
	
	We examine the ground-state energy recovery for the TFI model on an $N \times 1$ lattice, for system sizes $N \in \{64, 128, 256, 512, 1024, 2048, 4096\}$ and an external magnetic field strength parameter $h \in \{0.1, 1, 1.5\}$. Let $E_0$ denote the true ground-state energy and $\tilde{E}_0$ the lower bound of the ground-state energy obtained from (\textbf{P}). The relative error is defined as:
	$$\text{Err}_{\text{rel}} := \frac{E_0 - \tilde{E}_0}{|E_0|}.$$ The relative errors for Algorithms \ref{dual_problem_alm_hierarchies} and \ref{dual_problem_alm_two_hierarchies} are given in Tables \ref{energy_rel_err2} and \ref{energy_rel_err3}. Additionally, we present the evolution of our convergence metrics as a function of the ALM iteration number in Figures \ref{convergence_one_hierarchy} and \ref{convergence_two_hierarchies}, focusing on the $1$-D TFI model with a fixed external magnetic field strength parameter $h=1$ and various system sizes. Alongside the relative primal and dual feasibility measures and the relative duality gap, we also track the per-site primal objective change between subsequent iterations. All convergence metrics are transformed using a base-$10$ logarithm function. Within $150$ ALM iterations, all metrics drop below $10^{-3}$ for experiments with an external magnetic field strength parameter $h=1$.
	
	\begin{table}[ht!]
		\centering
		\begin{tabular}{|c|c|c|c|c|c|c|c|} 
			\hline
			& N=64 & N=128 & N=256 & N=512 & N=1024 & N=2048 & N=4096\\ [0.5ex] 
			\hline
			$h=0.5$ & $1.11\%$ & $1.12\%$ & $1.14\%$ & $1.14\%$ & $1.14\%$ & $1.15\%$ & $1.15\%$ \\ [1ex] 
			\hline
			$h=1$   & $2.73\%$ & $2.74\%$ & $2.75\%$ & $2.76\%$ & $2.77\%$ & $2.77\%$ & $2.77\%$\\ [1ex] 
			\hline
			$h=1.5$ & $0.71\%$ & $0.72\%$ & $0.72\%$ & $0.71\%$ & $0.71\%$ & $0.71\%$  & $0.71\%$\\ [1ex] 
			\hline
		\end{tabular}
		\caption{Relative errors of the ground-state energy from Algorithm \ref{dual_problem_alm_hierarchies}}
		\label{energy_rel_err2}
	\end{table}
	
	\begin{table}[h!]
		\centering
		\begin{tabular}{|c|c|c|c|c|c|c|c|} 
			\hline
			& N=64 & N=128 & N=256 & N=512 & N=1024 & N=2048 & N=4096\\ [0.5ex] 
			\hline
			$h=0.5$ & $1.10\%$ & $1.16\%$ & $1.16\%$ & $1.15\%$ & $1.16\%$ & $1.16\%$ & $1.17\%$ \\ [1ex] 
			\hline
			$h=1$   & $2.78\%$ & $2.75\%$ & $2.77\%$ & $2.78\%$ & $2.77\%$ & $2.77\%$ & $2.80\%$\\ [1ex] 
			\hline
			$h=1.5$ & $0.70\%$ & $0.71\%$ & $0.70\%$ & $0.70\%$ & $0.70\%$ & $0.71\%$ & $0.71\%$ \\ [1ex] 
			\hline
		\end{tabular}
		\caption{Relative errors of the ground-state energy from Algorithm \ref{dual_problem_alm_two_hierarchies}}
		\label{energy_rel_err3}
	\end{table}

	\begin{figure}
		\centering
		\includegraphics[width=0.9\textwidth]{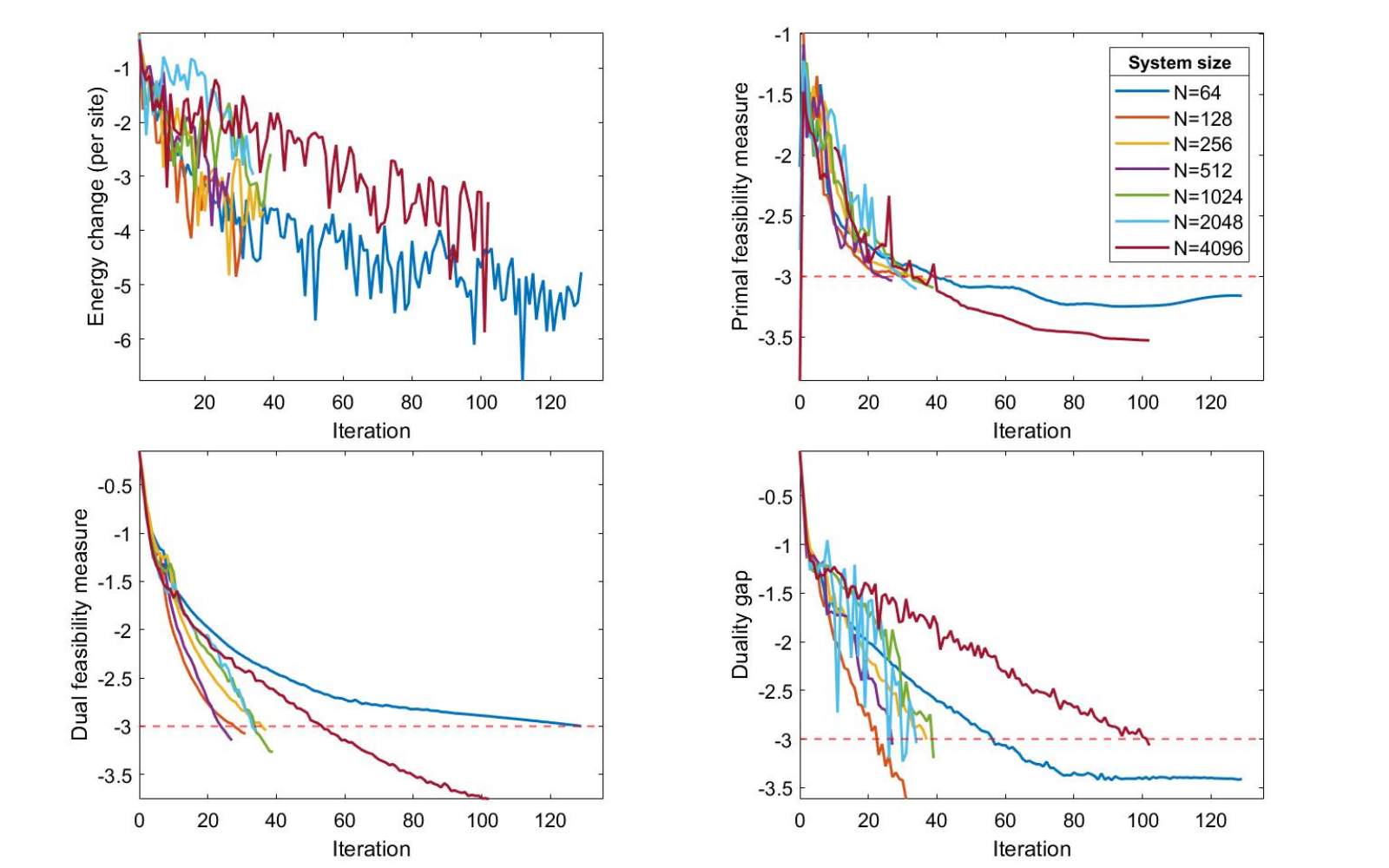}
		\caption{Convergence plot (in $\log_{10}$) for Algorithm \ref{dual_problem_alm_hierarchies}}
		\label{convergence_one_hierarchy}
	\end{figure}
	
	\begin{figure}
		\centering
		\includegraphics[width=0.9\textwidth]{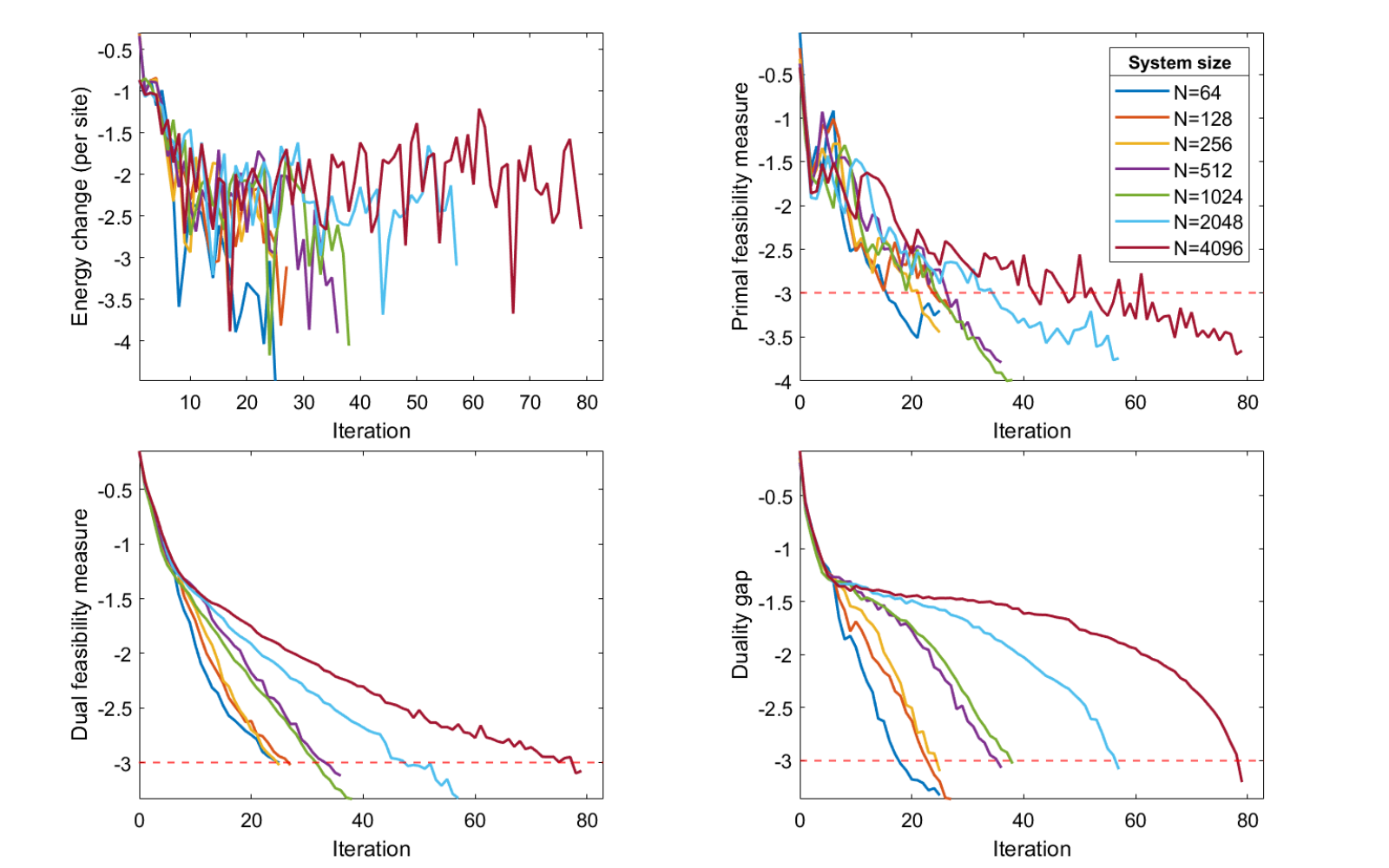}
		\caption{Convergence plot (in $\log_{10}$) for Algorithm \ref{dual_problem_alm_two_hierarchies}.}
		\label{convergence_two_hierarchies}
	\end{figure}

	\section{Conclusion}
	In this paper, we explored the computation of a specific semidefinite relaxation for determining the ground-state energy of many-body problems, which can be solved in polynomial time and provides a reasonable lower bound for the ground-state energy. Additionally, we identified a hierarchical structure in both the positive semidefinite (PSD) primal and dual variables, allowing us to circumvent the expensive projection onto the PSD cone, thereby reducing the per-iteration complexity of the ALM-type algorithm from cubic to quadratic or almost linear. 
	
	The relaxed problem provides only a lower bound for the lowest energy. To evaluate the effectiveness of our approach, we compare the recovered lower bound with the true ground-state energy for the $1$-D transverse field Ising model. Notably, for the most challenging case of $h=1$, where the system undergoes a quantum phase transition, our algorithm still produces a reasonable lower bound. Furthermore, our algorithm can handle systems consisting of up to $4096$ spins, whereas previous work on the variational embedding method, such as \cite{variational_embedding_qmb,variational_embedding_pd_algorithm}, which employs more accurate yet more expensive constraints, can only manage systems of a few dozen spins without leveraging the periodicity of the model.
	
	Currently, unlike \cite{variational_embedding_qmb,variational_embedding_pd_algorithm}, we do not impose conic constraints on each of the off-diagonal blocks of the moment matrix. Including them while leveraging the computational benefits of a hierarchical matrix will be the subject of future investigations. 

    \section*{Acknowledgements}
        Y.W. acknowledges partial support by ASCR Award DE-SC0022232 from the Department of Energy. Y.K. acknowledges partial support by NSF-2111563, NSF-2339439 from the National Science Foundation, and a Sloan research fellowship. Y.K. also acknowledges various interesting discussions with Ling Liang, Michael Lindsey, and Kim-Chuan Toh to speed up semidefinite programming.

	\newpage
	\appendix
	\section{Definitions of $\cA_U,\cA_L,\cD_U,\cD_L,\cD, J, w$ and $z$}

	In this section, we provide the definitions of the operators $\cA_U,\cA_L,\cD_U,\cD_L,\cD$, the matrix $J$, and the vectors $w$ and $z$ used in the main text. 
	\begin{itemize}
		\item The definition of the operator $\mathcal{A}_U:\mathbb{C}^{3\times 3}\to\mathbb{C}^9$ is given by:
		$$\mathcal{A}_U=A_U\circ \textcolor{blue}{vec},\quad A_U=\frac{1}{2i}I_9,$$
		where $\circ$ denotes the composition of the operators, $\textcolor{blue}{vec}:\mathbb{C}^{3\times 3}\to\mathbb{C}^9$ is the vectorization operator that stacks the columns of a matrix, $i$ is the imaginary unit, and $I_9$ is the $9\times 9$ identity matrix.
		\item The definition of the operator $\mathcal{A}_L:\mathbb{C}^{3\times 3}\to\mathbb{C}^9$ is given by:
		$$\mathcal{A}_L=A_L\circ \textcolor{blue}{vec},\quad A_L=-\frac{1}{2i}\begin{bmatrix}
        1&0&0&0&0&0&0&0&0\\
        0&0&0&1&0&0&0&0&0\\
        0&0&0&0&0&0&1&0&0\\
        0&1&0&0&0&0&0&0&0\\
        0&0&0&0&1&0&0&0&0\\
        0&0&0&0&0&0&0&1&0\\
        0&0&1&0&0&0&0&0&0\\
        0&0&0&0&0&1&0&0&0\\
        0&0&0&0&0&0&0&0&1\\
    \end{bmatrix}.$$
	\item The definition of the operator $\cD:\mathbb{C}^{3\times 3}\to\mathbb{C}^{12}$ is given by:
	$$\mathcal{D}=D\circ \textcolor{blue}{vec},\quad D=\begin{bmatrix}
        I_9\\O_{3\times 9}
    \end{bmatrix},$$
	where $O_{3\times 9}$ is a $3\times 9$ zero matrix.
	\item The definition of the operator $\mathcal{D}_U:\mathbb{C}^{3}\to\mathbb{C}^{12}$ is given by:
	$$\mathcal{D}_U=\frac{1}{2}\begin{bmatrix}
        0&0&0\\0&0&i\\0&-i&0\\0&0&-i\\0&0&0\\i&0&0\\0&i&0\\-i&0&0\\0&0&0\\1&0&0\\0&1&0\\0&0&1
    \end{bmatrix},$$
	\item The definition of the operator $\mathcal{D}_L:\mathbb{C}^{3}\to\mathbb{C}^{12}$ is given by:
	$$\mathcal{D}_L=D_L\circ H,\quad D_L=\frac{1}{2}\begin{bmatrix}
        0&0&0\\0&0&i\\0&-i&0\\0&0&-i\\0&0&0\\i&0&0\\0&i&0\\-i&0&0\\0&0&0\\-1&0&0\\0&-1&0\\0&0&-1
    \end{bmatrix},$$
	where $H:\mathbb{C}^{1\times 3}\to\mathbb{C}^{3\times 1}, v\mapsto v^*$ is the conjugate transpose operator.
	\item Given the parameter $h$ that controls the strength of the external
	magnetic field along the $x$ axis and the system size $N$, the matrix $J$ can be constructed by following rule:
	$$J(i,j)=\begin{cases}
		-0.5, & \text{if }\begin{matrix}
		& i=3N,\, j=3,\\
		& i=3,\, j=3N,\\
		& i=3k,\, j=3k+3,  \text{for } k\in[N-1],\\
		& i=3k+3,\, j=3k,  \text{for } k\in[N-1],
		\end{matrix}\\
		-0.5h & \text{if}\begin{matrix}
		& i=3k-2,\, j=3N+1,  \text{for } k\in[N],\\
		& i=3N+1,\, j=3k-2,  \text{for } k\in[N],
		\end{matrix}\\
		0, & \text{otherwise.}
		\end{cases}$$
	\item The vector $w$ is defined as:
	$$w=(0,0,0,0,0,0,0,0,0)^\top.$$
	\item The vector $z$ is defined as:
	$$z=(1,0,0,0,1,0,0,0,1,0,0,0)^\top.$$
	\end{itemize}

	\section{Complexity of operations with hierarchical matrices} 
	
	In this section, we present several propositions that describe the complexity of manipulating hierarchical matrices.
	
	\begin{proposition} \label{prop:computational_complexity_of_upper triangular}
		\begin{enumerate}
			\item Let $A=a_1 a_2^*$ and $B=b_1 b_2^*$ be two low-rank matrices with $a_1, a_2, b_1, b_2\in \mathbb{C}^{p\times r}$. Then $\sum_{i<j\leq p}A(i,j)B(i,j)$ has a complexity of $O(pr^2)$.
			\item Let $A=a_1 a_2^*$ be a low-rank matrix with $a_1, a_2\in \mathbb{C}^{p\times r}$ and $B\in \mathbb{C}^{p\times p}$. Then $\sum_{i<j\leq p}A(i,j)B(i,j)$ has a complexity of $O(p^2 r)$.
		\end{enumerate}
	\end{proposition}
	Based on this, we have the following propositions:
	\begin{proposition} \label{computational_complexity_hierarchical_inner_product}
		For two hierarchical matrices $A =\mathcal{H}^{(2)}(y)$ and $A' =\mathcal{H}^{(2)}(y')$ with $y = \{y^{(1)}, \cdots, y^{(m)}\}$ and ${y'} = \{{y'}^{(1)}, \cdots, {y'}^{(m)}\}$, where $y^{(l)}, {y'}^{(l')} \in \C^{CK \times r}$ for $1 \leq l,l' \leq m$, the formula 
		\begin{equation}
		\sum_{i<j\leq K} A_{ij}(k,\kappa) A_{ij}'(k',\kappa'),\quad \forall k,k',\kappa,\kappa'
		\end{equation}
		can be computed with $O(Km^2 r^2)$ complexity. 
	\end{proposition}
	\begin{proof}
		\begin{eqnarray}
		&\ & \sum_{i<j\leq K} A_{ij}(k,\kappa) A_{ij}'(k',\kappa')\cr 
		&=&\sum^m_{l=1}\sum^m_{l'=1}  \sum_{i<j\leq K}\left(\cH_{ij}^{(2)}(y^{(l)})(k,\kappa) \; \cH_{ij}^{(2)}({y'}^{(l')})(k',\kappa')\right).
		\end{eqnarray}
		For fixed $k$, $k'$, $\kappa$ and $\kappa'$, $[\cH_{ij}^{(2)}(y^{(l)})(k,\kappa)]_{i,j\in [K]}$ and $[\cH_{ij}^{(2)}({y'}^{(l')})(k',\kappa')]_{i,j\in [K]}$ take the form of 
		\begin{equation}
		B = \begin{pmatrix}
		B_1^{(l)} \\
		& B_2^{(l)} \\
		& & \ddots \\
		& & & B_{n_l}^{(l)}
		\end{pmatrix},\quad B' = \begin{pmatrix}
		{B'}_1^{(l')} \\
		& {B'}_2^{(l')} \\
		& & \ddots \\
		& & & {B'}_{n_{l'}}^{(l')}
		\end{pmatrix}
		\end{equation}
		respectively, where the diagonal blocks $B_i^{(l)}\in \mathbb{C}^{K/n_l\times K/n_l}$ and ${B'}_j^{(l')}\in \mathbb{C}^{K/n_{l'}\times K/n_{l'}}$ are rank-$r$ matrices. The complexity of $\sum_{i<j}\left(\cH_{ij}^{(2)}(y^{(l)})(k,\kappa) \; \cH_{ij}^{(2)}({y'}^{(l')})(k',\kappa')\right)$ can thus be obtained by applying Proposition~\ref{prop:computational_complexity_of_upper triangular} (Part 1) for $\max\{n_l,n_{l'}\}$ times, each time having $O(K/\max\{n_l,n_{l'}\}r^2)$ complexity, which gives a total complexity of $O(K r^2)$. Incorporating the double sum $\sum^m_{l=1}\sum^m_{l'=1}$ gives a final complexity of $O(Km^2r^2)$
	\end{proof}

	\begin{proposition} \label{prop:computational_complexity_hierarchical_inner_product_full}
		For a hierarchical matrix $A =\mathcal{H}^{(2)}(y)$ with $y = \{y^{(1)}, \cdots, y^{(m)}\}$ where $y^{(l)} \in \C^{CK \times r}$ for $1 \leq l \leq m$, and a dense matrix $A'\in \mathbb{C}^{CK\times CK}$, the formula 
		\begin{equation}
		\sum_{i<j\leq K} A_{ij}(k,\kappa) A_{ij}'(k',\kappa'),\quad \forall k,k',\kappa,\kappa'
		\end{equation}
		can be computed with $O(K^2 r)$ complexity. 
	\end{proposition}
	\begin{proof}
		\begin{eqnarray}
            % ,\quad \forall k,k',\kappa,\kappa'
		&\ & \sum_{i<j\leq K} A_{ij}(k,\kappa) A_{ij}'(k',\kappa') \cr
		&=&\sum^m_{l=1}  \sum_{i<j\leq K}\left(\cH_{ij}^{(2)}(y^{(l)})(k,\kappa) \; A_{ij}'(k',\kappa')\right).
		\end{eqnarray}
		For fixed $k$ and $\kappa$, $[\cH_{ij}^{(2)}(y^{(l)})(k,\kappa)]_{i,j\in [K]}$ takes the form of 
		\begin{equation}
		B = \begin{pmatrix}
		B_1^{(l)} \\
		& B_2^{(l)} \\
		& & \ddots \\
		& & & B_{n_l}^{(l)}
		\end{pmatrix},
		\end{equation}
		with some matrices $B^{(l)}_i \in \mathbb{C}^{K/n_l \times K/n_l}$ each having rank $r$. With this form, the complexity of \\ $\sum_{i<j\leq K} \cH_{ij}^{(2)}(y^{(l)})(k,\kappa) \; A_{ij}'(k',\kappa')$ can be determined by applying Proposition~\ref{prop:computational_complexity_of_upper triangular} (Part 2) $n_l$ times, each time having complexity $O(rK^2/n_l^2 )$, which gives a total complexity of $O(rK^2/n_l)$. Summing this complexity $\sum^m_{l=1}O(rK^2/n_l)$, we arrive at a final complexity of $O(K^2r)$.
	\end{proof}

	\bibliographystyle{plain}
	\bibliography{ref2}
\end{document}